\documentclass{amsart}

\usepackage{graphicx}
\usepackage[curve,frame,line,arrow,matrix]{xy}
\usepackage{sseq}
\usepackage{spectralsequences}

\usepackage{amssymb}   
\usepackage{amsmath}
\usepackage{textcomp}
\usepackage{tikz} 
\usepackage{tikz-cd}

\usepackage{hyperref}
\allowdisplaybreaks

\usepackage{amsthm}

\newtheorem{theorem}{Theorem}
\newtheorem{lemma}[theorem]{Lemma}

\theoremstyle{definition}
\newtheorem{definition}[theorem]{Definition}

\newtheorem{example}[theorem]{Example}
\newtheorem{examples}[theorem]{Examples}

\theoremstyle{conjecture}
\newtheorem{conjecture}[theorem]{Conjecture}

\theoremstyle{remark}
\newtheorem{remark}[theorem]{Remark}

\numberwithin{equation}{section}

\begin{document}

\title[Cohomology of Equivariant Eilenberg-Mac Lane Spaces]{$RO(C_2)$-graded Cohomology of $C_2$-equivariant Eilenberg-Mac Lane spaces}

\author{U\u{G}UR Y\.{I}\u{G}\.{I}T}
\address{Istanbul Medeniyet University, Istanbul, TURKEY}
\email{ugur.yigit@medeniyet.edu.tr}
\curraddr{Department of Mathematics, Istanbul Medeniyet University, H1-20, Istanbul, TURKEY 34700}
\thanks{}
\subjclass[2000]{}
\date{}
\dedicatory{}

\begin{abstract}
In this paper, we calculate $RO(C_2)$-graded cohomology of $C_2$-equivariant Eilenberg-Mac Lane spaces $K(\underline{Z/2}, n+\sigma)$ for $n\geq 0$. These can be used to give the relation between equivariant lambda algebra and equivariant Adams resolution and equivariant unstable Adams spectral sequence, which are defined in author`s dissertation. 
\end{abstract}


\maketitle

\tableofcontents

\begin{section}{Introduction}

An ordinary cohomology theory $H_G^{\star}(-: \underline{M})$ on $G$-spaces with Mackey functor $\underline{M}$ coefficients and graded by real orthogonal representations is defined by Lewis, May and Mcclure \cite{LMM1981}. In this paper, we compute the $RO(C_2)$-graded cohomology of the $C_2$-equivariant Eilenberg-Mac Lane spaces with the constant Mackey functor $\underline{M}=\underline{Z/2}$ coefficients, which are crucial to give the relation between the equivariant lambda algebra and the equivariant  unstable Adams resolution and equivariant unstable Adams spectral sequence, which is given by Mahowald \cite{Mah:ImJ} in the classical case. Throughout this paper, $H^{\star}(-)$ denotes the ordinary $RO(C_2)$-graded cohomology of a $C_2$-space with the constant Mackey functor coefficients $\underline{\mathbb{Z}/2}$.

To compute the $RO(C_2)$-graded cohomology of the $C_2$-equivariant Eilenberg-Mac Lane spaces with the constant Mackey functor $\underline{M}=\underline{Z/2}$ coefficients, we use Borel theorem \ref{Bor} for the path-space fibration
$$\Omega K(\underline{Z/2}, V)\longrightarrow P(K(\underline{Z/2}, V))\longrightarrow K(\underline{Z/2}, V).$$
for $V=\sigma+n$, where $n\geq 0$.

If we knew $H^{\star}( K(\underline{Z/2}, n\sigma))$ for $n\geq 2$, one could use the Eilenberg-Moore spectral sequence \cite[Chapter 5]{Hillfree}, the Borel theorem, and the $RO(G)$-graded Serre spectral sequence of Kronholm \cite[Theorem 1.2.]{WKronholm2010-2} for the path-space fibration
$$\Omega K(\underline{Z/2}, V)\longrightarrow P(K(\underline{Z/2}, V))\longrightarrow K(\underline{Z/2}, V).$$

This paper is organized as follows. In section \ref{s2}, we provide the basic equivariant topology tools, and $C_2$-equivariant cohomology $M^{C_2}_2$ of a point, and equivariant connectivity of $G$-spaces. In section \ref{s3}, we descripe equivariant Steenrod squares, $C_2$-equivariant Steenrod algebra $\mathcal{A}_{C_2}$ and axioms of it. In section \ref{s4}, we give the definition of the equivariant Eilenberg-Mac Lane spaces with some properties, and the fixed point sets of the equivariant Eilenberg-Mac Lane spaces that is very useful to compute the cohomology of them. In section \ref{s5}, we compute the $RO(C_2)$-graded $C_2$-equivariant cohomology of some $C_2$-equivariant Eilenberg-Mac Lane spaces $K_V$ for real orthogonal representations $V=\sigma+n, \ n\geq 0$. Also, we give some conjectures and future directions for the other cases.\\
\textbf{Notation.} We provide here notation used in this paper for convenience.
\begin{itemize}
\item[$\bullet$]$V=r\sigma+s$, a real orthogonal representation of $C_2$, which is a sum of $r$-copy of the sign representation $\sigma$ and $s$-copy of the trivial representation $1$.
\item[$\bullet$]$\rho=\sigma+1$, the regular representation of $C_2$.
\item[$\bullet$]$RO(C_2)$, the real representation ring of $C_2$.
\item[$\bullet$]$S^V$, the equivariant sphere which is the one-point compactification of $V$.
\item[$\bullet$]$\pi_{V}^{C_2}(X)$, the $V$-th $C_2$-equivariant homotopy group of a topological $C_2$-space $X$.
\item[$\bullet$]$\pi_{r\sigma+s}^S$, the $C_2$-equivariant stable homotopy groups of spheres.
\item[$\bullet$]$\Sigma^{\sigma}(X)$, the $\sigma$-th suspension of $X$.
\item[$\bullet$]$\Omega^{\sigma}(X)$, all continuous functions from $S^{\sigma}$ to $X$.
\item[$\bullet$]$H_G^{\star}(-: \underline{M})$, RO($G$)-graded ordinary equivariant cohomology with Mackey functor $\underline{M}$ coefficients.
\item[$\bullet$]$M^{C_2}_2$, RO($C_2$)-graded $C_2$-equivariant cohomology of a point.
\item[$\bullet$]$\mathcal{A}_{C_2}$, $C_2$-equivariant Steenrod algebra.
\item[$\bullet$]$K(\underline{M}, V)$ or shortly $K_V$, the $V$th equivariant Eilenberg-Mac Lane space with a Mackey functor $\underline{M}$.
\item[$\bullet$]$\underline{\pi}_{V}^G(X)$, $C_2$-equivariant homotopy of a $G$-space $X$ as a Mackey functor.
\item[$\bullet$] $Sq_{C_2}^{k}$, $C_2$-equivariant Steenrod squaring operations for $k\geq 0$.
\item[$\bullet$] $RP^{\infty}_{tw}$, the space of lines in the
complete universe $\mathcal{U}=(\mathbb{R}^{\rho})^{\infty}$, which is equivalent to $K(\underline{Z/2}, \sigma)$.
\end{itemize}
\end{section}
\textbf{Acknowledgements.} I would like to thank Michael A. Hill for valuable conversations and providing me some suggestions for calculations, and William Kronholm for producing the action of Steenrod squares on the cohomology ring of $RP_{tw}^{\infty}$. Lastly and most importanly, I would like to state my gratitude to my advisor, Douglas C. Ravenel, for his patience, support, and encouragement throughout my graduate studies, and numerous beneficial conversations and suggestions. The work in this paper was part of the author’s dissertation while at the University of Rochester.

\begin{section}{Preliminaries}\label{s2}

In this section we give the main tools that are used in the rest of the article. Let $X$ be a $G$-space, where $G=C_2$ is a cyclic group with generator $\gamma$ such that $\gamma^2=e$. The group $C_2$ has two irreducible real representations, namely the trivial representation denoted by $1$ (or $\mathbb{R}$)
and the sign representation denoted by $\sigma$ (or $\mathbb{R_-}$). The regular representation is isomorphic to $\rho_{C_2}=1+\sigma$ (it is denoted by $\rho$ if there is no confusion). Thus the representation ring $RO(C_2)$ is free abelian of rank $2$, so every representation $V$ can be expressed as $V=r\sigma+s$.

\begin{definition}\label{Uni}A \textbf{$G$-universe} is a countably infinite-dimensional $G$-representation which contains the trivial
$G$-representation and which contains infinitely many copies of each of its
finite-dimensional subrepresentations. Also, a \textbf{ complete G-universe}
is just a G-universe that contains infinitely many copies of every irreducible
$G$-representation.
\end{definition}
\begin{definition} A $G$-spectrum $E$ on a $G$-universe $\mathcal{U}$ is a collection $E_V$ of based $G$-spaces together with basepoint-preserving $G$-maps
$$\sigma_{V,W}:\Sigma^{W-V}E_V\longrightarrow E_W$$
whenever $V\subset W\subset \mathcal{U}$, where $W-V$ denotes the orthogonal complement of $V$ in $W$. It is required that $\sigma_{V,V}$ is identity, and the commutativity of the diagram

\[
\begin{tikzcd}
    \Sigma^{W-V}\Sigma^{V-U}E_U \arrow{rr}{\Sigma_{W-V}\sigma_{U,V}} \arrow[swap]{dr}{\sigma_{U,W}} & & \Sigma^{W-V}E_V \arrow{dl}{\sigma_{V,W}} \\[10pt]
    & E_W 
\end{tikzcd}
\]
for $U\subset V\subset W\subset \mathcal{U}$.
\end{definition}
\begin{definition} If the adjoint structure maps
$$\tilde{\sigma}_{V,W}
: E_V\longrightarrow \Omega^{W-V}E_W$$
are weak homotopy equivalences for $V\subset W\subset \mathcal{U}$, then a $G$-spectrum is called $G-\Omega$-spectrum.
\end{definition}

A $G$-spectrum indexed on a complete(trivial) $G$-universe is called genuine(naive).

For an actual representation $V$ of $G$ and a $G$-space $X$, the $V$-th
homotopy group of $X$ is the Mackey functor $\underline{\pi}_V (X)$ determined by
$$\underline{\pi}_V (X)(G/H)=[S^V,X]^{H}$$
for every $H<G$.

For a virtual representation $V\in RO(G)$ and a $G$-spectrum $E$, the
$V$-th homotopy group of $E$ is the Mackey functor $\underline{\pi}_V(E)$ determined by
$$\underline{\pi}_V(E)(G/H) = colim_{n}
\pi_0(\Omega^{V+W_n}E_{W_n})^H
$$
where $\{W_n|n\in \mathbb{N}\}$ is an increasing sequence of representations $$\cdots\subset W_n\subset W_{n+1}\subset\cdots$$
such that any finite dimensional representation
$V$ of $G$ admits an equivariant embedding in some $W_n$.

Lewis, May and Mcclure \cite{LMM1981} defined an ordinary cohomology theory $H_G^{\star}(-: \underline{M})$ on $G$-spaces with Mackey functor $\underline{M}$ coefficients and the graded by real orthogonal representations.

Throughout this paper, the Mackey functor will typically be the constant Mackey functor $\underline{M}=\underline{Z/2}$, which can be given the following diagram in Lewis notation.
\begin{equation}\label{ConsMac}
\xymatrix
@R=7mm
@C=10mm{
\mathbb{Z}/2 \ar@/_/[d]_{Id}\\
\mathbb{Z}/2 \ar@(dr,dl)[]^{Id} \ar@/_/[u]_{0}\\
}
\end{equation}

The ordinary equivariant cohomology $M^{C_2}_2$ of a point with this coefficient is given in the Figure \ref{fig:pt} below. Every $\bullet$ in the figure represents a copy of $Z/2$.

As you see in the Figure \ref{fig:pt} below, there are two elements of interest. The inclusion map of the fixed point set (the north and
south poles) $a: S^0\longrightarrow S^{\sigma}$ defines an element in $\pi_{-\sigma}^{C_2}(S^{-0})$,
and we will use the same symbol for its mod $2$ Hurewicz
image. It is called an Euler class. One can show that
$$H_1^{C_2}(S^{\sigma};Z/2)= H^{C_2}_{1-\sigma}(S^{-0};Z/2)=Z/2$$
and we denote its generator by $u$. Dually, we have $a\in H_{C_2}^{\sigma}(S^{-0};Z/2)$ and $u\in H_{C_2}^{\sigma-1}(S^{-0};Z/2)$. These are the analog of elements $\rho$ and $\tau$ in real motivic homotopy theory, respectively.

\begin{figure}[h]
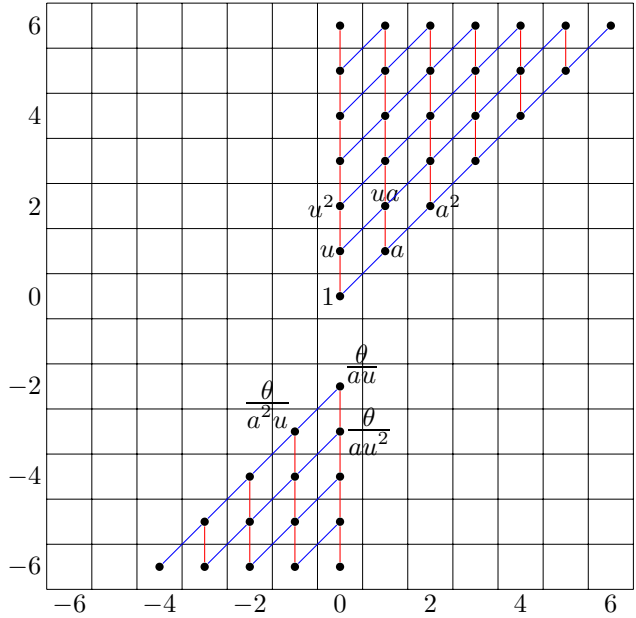

\centering
\begin{sseq}[grid=crossword, entrysize=6mm]{-6...6}{-6...6}
\ssdropbull
\ssdroplabel{1}
\ssline[color=red]{0}{1}
\ssdrop{\bullet}
\ssdroplabel{u}
\ssline[color=red]{0}{1}
\ssdrop{\bullet}
\ssdroplabel{u^2}
\ssline[color=red]{0}{1}
\ssdrop{\bullet}
\ssline[color=red]{0}{1}
\ssdrop{\bullet}
\ssline[color=red]{0}{1}
\ssdrop{\bullet}
\ssline[color=red]{0}{1}
\ssdrop{\bullet}

\ssmove {1}{-5}
\ssdrop{\bullet}
\ssdroplabel[R]{a}
\ssline[color=red]{0}{1}
\ssdrop{\bullet}
\ssdroplabel[U]{ua}
\ssline[color=red]{0}{1}
\ssdrop{\bullet}
\ssline[color=red]{0}{1}
\ssdrop{\bullet}
\ssline[color=red]{0}{1}
\ssdrop{\bullet}
\ssline[color=red]{0}{1}
\ssdrop{\bullet}

\ssmove {1}{-4}
\ssdrop{\bullet}
\ssdroplabel[R]{a^2}
\ssline[color=red]{0}{1}
\ssdrop{\bullet}
\ssline[color=red]{0}{1}
\ssdrop{\bullet}
\ssline[color=red]{0}{1}
\ssdrop{\bullet}
\ssline[color=red]{0}{1}
\ssdrop{\bullet}

\ssmove {1}{-3}
\ssdrop{\bullet}
\ssline[color=red]{0}{1}
\ssdrop{\bullet}
\ssline[color=red]{0}{1}
\ssdrop{\bullet}
\ssline[color=red]{0}{1}
\ssdrop{\bullet}

\ssmove {1}{-2}
\ssdrop{\bullet}
\ssline[color=red]{0}{1}
\ssdrop{\bullet}
\ssline[color=red]{0}{1}
\ssdrop{\bullet}

\ssmove {1}{-1}
\ssdrop{\bullet}
\ssline[color=red]{0}{1}
\ssdrop{\bullet}
\ssmove {1}{0}
\ssdrop{\bullet}

\ssmove {-6} {-6}
\ssline[color=blue]{1}{1}
\ssline[color=blue]{1}{1}
\ssline[color=blue]{1}{1}
\ssline[color=blue]{1}{1}
\ssline[color=blue]{1}{1}
\ssline[color=blue]{1}{1}

\ssmove {-6} {-5}
\ssline[color=blue]{1}{1}
\ssline[color=blue]{1}{1}
\ssline[color=blue]{1}{1}
\ssline[color=blue]{1}{1}
\ssline[color=blue]{1}{1}

\ssmove {-5} {-4}
\ssline[color=blue]{1}{1}
\ssline[color=blue]{1}{1}
\ssline[color=blue]{1}{1}
\ssline[color=blue]{1}{1}

\ssmove {-4} {-3}
\ssline[color=blue]{1}{1}
\ssline[color=blue]{1}{1}
\ssline[color=blue]{1}{1}

\ssmove {-3} {-2}
\ssline[color=blue]{1}{1}
\ssline[color=blue]{1}{1}

\ssmove {-2} {-1}
\ssline[color=blue]{1}{1}

\ssmove {-1}{-8}
\ssdrop{\bullet}
\ssdroplabel[RU]{\frac{\displaystyle\theta}{\displaystyle au}}

\ssline[color=red]{0}{-1}
\ssdrop{\bullet}
\ssdroplabel[R]{\frac{\displaystyle\theta}{\displaystyle au^2}}
\ssline[color=red]{0}{-1}
\ssdrop{\bullet}
\ssline[color=red]{0}{-1}
\ssdrop{\bullet}
\ssline[color=red]{0}{-1}
\ssdrop{\bullet}

\ssmove {-1}{3}
\ssdrop{\bullet}
\ssdroplabel[LU]{\frac{\displaystyle\theta}{\displaystyle a^2u}}
\ssline[color=red]{0}{-1}
\ssdrop{\bullet}
\ssline[color=red]{0}{-1}
\ssdrop{\bullet}
\ssline[color=red]{0}{-1}
\ssdrop{\bullet}

\ssmove {-1}{2}
\ssdrop{\bullet}
\ssline[color=red]{0}{-1}
\ssdrop{\bullet}
\ssline[color=red]{0}{-1}
\ssdrop{\bullet}

\ssmove {-1}{1}
\ssdrop{\bullet}
\ssline[color=red]{0}{-1}
\ssdrop{\bullet}
\ssmove {-1}{0}
\ssdrop{\bullet}

\ssmove {4}{4}
\ssline[color=blue]{-1}{-1}
\ssline[color=blue]{-1}{-1}
\ssline[color=blue]{-1}{-1}
\ssline[color=blue]{-1}{-1}

\ssmove {4}{3}
\ssline[color=blue]{-1}{-1}
\ssline[color=blue]{-1}{-1}
\ssline[color=blue]{-1}{-1}

\ssmove {3}{2}
\ssline[color=blue]{-1}{-1}
\ssline[color=blue]{-1}{-1}

\ssmove {2}{1}
\ssline[color=blue]{-1}{-1}

\end{sseq}

\caption{The equivariant cohomology $M_{2}^{C_2}$ of a point} \label{fig:pt}
\end{figure}

The coordinate $(x, y)$ represents degree $(x-y)+\sigma y$, which is convenient with the motivic bidegree. Red and blue lines represent multiplication by $u$ and $a$, respectively. 

Now, we will give the definition of equivariant connectivity of $G$-spaces.
\begin{definition}\cite{Lewis92}
\begin{enumerate}
\item[(i)] A function $\nu^*$ from the set of conjugacy classes of subgroups of $G$ to the integers is called a \textbf{dimension function}. The value of $\nu^*$ on the
conjugacy class of $K \subset G$ is denoted by $\nu^K$. Let $\nu^*$ and $\mu^*$ be two dimension
functions. If $\nu^K \geq \mu^K$ for every subgroup $K$, then $\nu^*\geq \mu^*$. Associated to
any $G$-representation $V$ is the dimension function $|V^*|$ whose value at $K$ is
the real dimension of the $K$-fixed subspace $V^K$ of $V$. The dimension function
with constant integer value $n$ is denoted $n^*$ for any integer $n$.\\
\item[(ii)] Let $\nu^*$ be a non-negative dimension function. If for each subgroup $K$ of $G$, the fixed point space $Y^K$ is $\nu^K$-connected, then a $G$-space $Y$ is called \textbf{$G$-$\nu^*$-connected}. If A $G$-space $Y$ is $G$-$0^*$-connected, then it is called \textbf{$G$-connected}. Also, if it is $G$-$1^*$-connected, it is called \textbf{simply $G$-connected}. A $G$-space $Y$ is \textbf{homologically $G$-$\nu^*$-connected} if, for every subgroup $K$ of $G$ and every integer $m$
with $0\leq m\leq \nu^K$ , the homology group $H_m^K(Y)$ is zero.\\
\item[(iii)] Let $\nu^*$ be a non-negative dimension function and let $f:Y\longrightarrow Z$ be a
$G$-map. If, for every subgroup $K$ of $G$,
\begin{displaymath}
(f^K)_*:\pi_m(Y^k)\longrightarrow \pi_m(Z^K)
\end{displaymath}
is an isomorphism for every integer $m$ with $0\leq m < \nu^K$ and an epimorphism
for $m = \nu^K$, then $f$ is called \textbf{$G$-$\nu^*$-equivalence}. A $G$-pair $(Y, B)$ is said to be \textbf{$G$-$\nu^*$-connected} if the inclusion of
$B$ into $Y$ is a $G$-$\nu^*$-equivalence. The notions of \textbf{homology $G$-$\nu^*$-equivalence}
and of \textbf{homology $G$-$\nu^*$-connectedness} for pairs are defined similarly, but with
homotopy groups replaced by homology groups.\\
\item[(iv)] Let $V$ be a $G$-representation. For each subgroup $K$ of $G$, let $V(K)$
be the orthogonal complement of $V^K$ ; then $V(K)$ is a $K$-representation. If  $\pi_{V(K)+m}^K(Y)$ is zero for each subgroup $K$ of $G$ and each integer $m$ with $0\leq m\leq |V^K|$, the
$G$-space $Y$ is called \textbf{$G$-$V$-connected}. Similarly, if $H_{V(K)+m}^G(Y)$ is zero for each subgroup $K$ of $G$ and each integer $m$ with $0\leq m\leq |V^K|$, then the $G$-space $Y$ is called \textbf{homologically $G$-$V$-connected}.\\
\item[(v)] Let $V$ be a $G$-representation. A $G$-$0^*$-equivalence $f: Y\longrightarrow Z$ is said to
be a \textbf{$G$-$V$-equivalence} if, for every subgroup $K$ of $G$, the map
\begin{displaymath}
f_*: \pi_{V(K)+m}^K(Y)\longrightarrow \pi_{V(K)+m}^K(Z)
\end{displaymath}
is an isomorphism for every integer $m$ with $0\leq m <|V^K|$ and an epimorphism
for $m = |V^K|$. A \textbf{homology $G$-$V$-equivalence} is defined similarly. A $G$-pair
$(Y, B)$ is called \textbf{$G$-$V$-connected} (respectively, \textbf{homologically $G$-$V$-connected}) if the inclusion of $B$ into $Y$ is a $G$-$V$-equivalence (respectively, homology $G$-$V$-equivalence).
\end{enumerate}
\end{definition}
  
\end{section}

\begin{section}{$C_2$-Equivariant Steenrod Algebra}\label{s3}

The analog of the mod $2$ Steenrod algebra is defined by Voevodsky \cite{Voe:Red} in the motivic case, and Po Hu and Igor Kriz \cite{HuKriz1} in the equivariant case. The two descriptions are essentially the same.

One has squaring operations $Sq^k_{C_2}$ for $k\geq 0$, whose degrees 

$$|Sq^k_{C_2}|=\left\{
\begin{array}{ll}
      i(1+\sigma) & for \ k=2i \\
      i(1+\sigma)+1 & for \ k=2i+1. \\
\end{array} 
\right.$$ \\

$Sq^0_{C_2}=1$ as in the classical case. The $C_2$-equivariant Steenrod algebra acts on the coefficient ring $M^{C_2}_2$ by
\begin{equation}\label{act}
Sq^k_{C_2}(u)=\left\{
\begin{array}{ll}
      u & for \ k=0 \\
      a & for \ k=1 \\
      0 & else. \\
\end{array} 
\right.
\end{equation}

\begin{equation}\label{act2}
Sq^{2m+\delta}_{C_2}(u^{2l+\epsilon})= \binom{2l+\epsilon}{2m+\delta} u^{2l+\epsilon-m-\delta}a^{2m+\delta}
\end{equation}
The difficulty in deriving the formula \ref{act2} is the $C_2$-equivariant Cartan formula \ref{car}, \ref{car2}. Since
$$|Sq_{C_2}^{2m+\delta}|=m(1+\sigma)+\delta \ \ \ \ \text{for} \ 0\leq \delta \leq 1,$$
we have
\begin{equation}\label{6}
\left\{
\begin{array}{ll}
      \Delta(Sq_{C_2}^{2m+1})=\sum_{0\leq i \leq 2m+1} Sq_{C_2}^i\otimes Sq_{C_2}^{2m+1-i}\\
      
     \\
      \Delta(Sq_{C_2}^{2m})=\sum_{0\leq j \leq m} Sq_{C_2}^{2i}\otimes Sq_{C_2}^{2m-2j}+ u \sum_{1\leq j \leq m} Sq_{C_2}^{2j-1}\otimes Sq_{C_2}^{2m-2j+1}.\\
\end{array} 
\right.
\end{equation}
The terms divisible by $u$ make things difficult. Here we are using cohomological degree, so $|u|=\sigma-1$. Note that
$$|u^{-m}Sq_{C_2}^{2m+\delta}|= m(1-\sigma)+m(1+\sigma)+\delta=2m+\delta$$
and define
$$\mathcal{S}q^{2m+\delta}:=u^{-m}Sq_{C_2}^{2m+\delta}.$$
We will see that these operations satisfy the classical Cartan formula. We have

\begin{align*}
\Delta(\mathcal{S}q^{2m+1})&=u^{-m}\Delta(Sq_{C_2}^{2m+1})\\
\
&=u^{-m}\sum_{0\leq i \leq 2m+1} Sq_{C_2}^i\otimes Sq_{C_2}^{2m+1-i}\\
\
&=\sum_{0\leq i \leq 2m+1} u^{-\lfloor i/2 \rfloor}Sq_{C_2}^i\otimes u^{-\lfloor (2m+1-i)/2 \rfloor}Sq_{C_2}^{2m+1-i}\\
\
&=\sum_{0\leq i \leq 2m+1} \mathcal{S}q_{C_2}^i\otimes \mathcal{S}q_{C_2}^{2m+1-i}
\end{align*}
$\text{since}$  $\lfloor i/2 \rfloor+ \lfloor (2m+1-i)/2 \rfloor=m$. And also,

\begin{align*}
\Delta(\mathcal{S}q^{2m})&=u^{-m}\Delta(Sq_{C_2}^{2m})\\
&=u^{-m}\sum_{0\leq j \leq m} Sq_{C_2}^{2j}\otimes Sq_{C_2}^{2m-2j}+u^{1-m}\sum_{1\leq j \leq m} Sq_{C_2}^{2j-1}\otimes Sq_{C_2}^{2m-2j+1}\\
&=\sum_{0\leq j \leq m} u^{-j}Sq_{C_2}^{2j}\otimes u^{j-m}Sq_{C_2}^{2m-2j}+\sum_{1\leq j \leq m} u^{1-j}Sq_{C_2}^{2j-1}\otimes u^{j-m}Sq_{C_2}^{2m-2j+1}\\
&=\sum_{0\leq j \leq m} \mathcal{S}q_{C_2}^{2j}\otimes \mathcal{S}q_{C_2}^{2m-2j}+\sum_{1\leq j \leq m}\mathcal{S}q_{C_2}^{2j-1}\otimes \mathcal{S}q_{C_2}^{2m-2j+1}\\
&=\sum_{0\leq i \leq 2m} \mathcal{S}q_{C_2}^i\otimes \mathcal{S}q_{C_2}^{2m-i}.
\end{align*}
Now, if we use homological degree, then
$$|\mathcal{S}q^m|=-m, |a|=-\sigma, \ \text{and} \ |u|=1-\sigma.$$
We know that 
\begin{equation}\label{act3}
Sq^m_{C_2}(u)=\left\{
\begin{array}{ll}
      u & for \ m=0 \\
      a & for \ m=1 \\
      0 & else. \\
\end{array} 
\right.
\end{equation}
Consider the total Steenrod operation
\begin{equation}\label{total}
\mathcal{S}q_t=\sum_{i\geq 0} t^{i}\mathcal{S}q^i,
\end{equation}
where $t$ is a dummy variable. Although this sum is infinite, it yields a finite sum when applied to any monomial in $a$ and $u$. The classical Cartan formula satisfied by operations $\mathcal{S}q^{i}$ implies that it is a ring homomorphism, meaning that
$$\mathcal{S}q_t(xy)=\mathcal{S}q_t(x) \mathcal{S}q_t(y).$$
Then \ref{act3} implies that
\begin{align*}
\mathcal{S}q_t(u)=u+ta
\end{align*}
\begin{align*}
\mathcal{S}q_t(u^l)&=(u+ta)^l\\
&=\sum_{0\leq m \leq l} \binom{l}{m}t^mu^{l-m}a^m\\
&=\sum_{0\leq m \leq l}t^m\mathcal{S}q^m(u^l).
\end{align*}
Hence, $\mathcal{S}q^{m}(u^l)$ is the coefficient of $t^m$ in the first sum above.

It follows that
\begin{align*}
\mathcal{S}q^{2m+\delta}(u^{2l+\epsilon})&=\binom{2l+\epsilon}{2m+\delta} u^{2l+\epsilon-2m-\delta} a^{2m+\delta}\\
\mathcal{S}q^{2m+\delta}(u^{2l+\epsilon})&=u^m\mathcal{S}q^{2m+\delta}(u^{2l+\epsilon})\\
&=\binom{2l+\epsilon}{2m+\delta} u^{2l+\epsilon-m-\delta} a^{2m+\delta}.
\end{align*}
As a result, we have the following:
\begin{lemma}\label{lem1}
$$
Sq^{2m+\delta}_{C_2}(u^{2l+\epsilon})= \binom{2l+\epsilon}{2m+\delta} u^{2l+\epsilon-m-\delta}a^{2m+\delta}.
$$
\end{lemma}

The natural action of the Steenrod algebra in homology is on the right, not on the left. Classically, the $\mod p$ cohomology of a space or a spectrum $X$ is a left module over the Steenrod algebra $\mathcal{A}$, so there is a map
$$c_X: A\otimes H^{\ast}X\rightarrow H^{\ast}X.$$
The Steenrod algebra has a multiplication
$$\phi^{\ast}:\mathcal{A}\otimes \mathcal{A}\rightarrow \mathcal{A}$$
(the symbol $\phi^{\ast}$ and its dual $\phi_{\ast}$ are taken from Milnor's paper \cite{MilA}) and the following diagram commutes
\begin{equation}\label{diag}
\begin{tikzcd}
 \mathcal{A}\otimes\mathcal{A}\otimes H^{\ast}X \ \ \arrow[r, "\phi^{\ast}\otimes H^{\ast}X"] \arrow[d, "\mathcal{A}\otimes c_X"]
& \mathcal{A}\otimes H^{\ast}X \arrow[d, "c_X"] \\
\mathcal{A}\otimes H^{\ast}X \arrow[r, "c_X"]
& H^{\ast}X.
\end{tikzcd}
\end{equation}

Milnor defines a right action of $\mathcal{A}$ on $H_{\ast}X$ by the rule
$$\langle xa, y \rangle=\langle x, ay \rangle\in \mathbb{F}_p$$
for $x\in H_{\ast}X, \ a\in \mathcal{A}$ and $y\in H^{\ast}X$, where the brackets denotes the evaluation of the cohomology class on the right on the homology class on the left. Milnor denotes by $\lambda^{\ast}$ the resulting map
$$H_{\ast}X\otimes \mathcal{A}\rightarrow H_{\ast}X.$$

The same thing happens in the $C_2$-equivariant case. For example, we have
$$(u^2)Sq_{C_2}^{3}=(u^2)Sq_{C_2}^{1}Sq_{C_2}^{2}=0$$
because $(u^2)Sq_{C_2}^{1}=0$. And,
$$(u^2)\chi(Sq_{C_2}^{3})=(u^2)Sq_{C_2}^{2}Sq_{C_2}^{1}=(ua^2)Sq_{C_2}^{1}=a^3,$$
where $\chi(-)$ means the conjugate Steenrod operations. Hence, \ref{act2} should really read as
$$(u^{2l+\epsilon})Sq^{2m+\delta}_{C_2}= \binom{2l+\epsilon}{2m+\delta} u^{2l+\epsilon-m-\delta}a^{2m+\delta}.$$

For example,
\begin{align*}
Sq^{l}_{C_2}(u^{-1}) &=\binom{-1}{l}a^lu^{-1-l}\\
&=\left\{
\begin{array}{ll}
      \binom{-1}{0}u^{-1}=u^{-1} & \text{for} \ l=0\\
      \\
      \binom{-1}{1}au^{-2}=au^{-2} &  \text{for} \ l=1\\
      \\
      \binom{-1}{2}a^2u^{-3}=a^2u^{-3} & \text{for} \ l=2\\
      \\
      0 & \text{for} \ l\geq 3\\
\end{array} 
\right.\\
\end{align*}

Action on the other elements is determined by the Cartan formula (iv) given below. We now give axioms for the squares $Sq^k_{C_2}$. For the motivic case, you can check Voevodsky paper \cite{Voe:Red}. But, the Adem relation is fixed by Jo\"el Riou in \cite{1207.3121}.
\begin{itemize}
\item[(i)] $Sq^0_{C_2}=1$ and $Sq^1_{C_2}=\beta_{C_2}$, Bockstein homomorphism.
\item[(ii)] $\beta Sq^{2k}_{C_2}=Sq^{2k+1}_{C_2}$.
\item[(iii)]$\beta Sq^{2k+1}_{C_2}=0$.
\item[(iv)](Cartan formula)\label{Car}
\begin{equation}\label{car}
Sq^{2k}_{C_2}(xy)= \sum_{r=0}^{k}Sq^{2r}_{C_2}(x)Sq^{2k-2r}_{C_2}(y)+u\sum_{s=0}^{k-1}Sq^{2s+1}_{C_2}(x)Sq^{2k-2s-1}_{C_2}(y)
\end{equation}
\begin{equation}\label{car2}
Sq^{2k+1}_{C_2}(xy)= \sum_{r=0}^{2k+1}Sq^{r}_{C_2}(x)Sq^{2k+1-r}_{C_2}(y)+a\sum_{s=0}^{k-1}Sq^{2s+1}_{C_2}(x)Sq^{2k-2s-1}_{C_2}(y)
\end{equation}
\item[(v)](Adem relation) If $0<i<2j$, then
when $i+j$ is even
$$Sq^i_{C_2}Sq^j_{C_2}=\sum_{k=0}^{[i/2]}\binom{b-1-k}{i-2k}u^{\epsilon}Sq^{i+j-k}_{C_2}Sq^k_{C_2}$$
where
$$\epsilon=\left\{
\begin{array}{ll}
      1 &  for \ k \ is \ odd \ and \ i \ and \ j \ are \ even\\
      0 &  else\\
\end{array} 
\right.$$
when $i+j$ is odd
$$Sq^i_{C_2}Sq^j_{C_2}=\sum_{k=0}^{[i/2]}\binom{j-1-k}{i-2k}Sq^{i+j-k}_{C_2}Sq^k_{C_2}+a\sum_{k= odd}\varepsilon \ Sq^{i+j-k}_{C_2}Sq^k_{C_2}$$
where
$$\varepsilon=\left\{
\begin{array}{ll}
      \binom{j-1-k}{i-2k} &  for \ i \ is \ odd\\
      \binom{j-1-k}{i-2k-1} &  for \ j \ is \ odd\\
\end{array} 
\right.$$
\item[(vi)] If $x$ has a degree $k\sigma+k$, then $Sq^{2k}_{C_2}(x)=x^2$.
\item[(vii)](instability) If $x$ has a degree $V$, $V<k \sigma+k$ then $Sq^{2k}_{C_2}(x)=0$, where $V<V^{'}$ if and only if $V^{'}=V+W$ for some actual representations $W$ with positive degree.
\end{itemize}
Note that setting $u= 1$ and $a= 0$ reduces the Cartan formula (iv) to the classical Cartan formula, and Adem relation (v) to the classical Adem relation.

\begin{examples} We have
$$Sq^1_{C_2}Sq^n_{C_2}=\left\{
\begin{array}{ll}
      Sq^{n+1}_{C_2} &  for \ n \ is \ even\\
      0 &  for \ n \ is \ odd\\
\end{array} 
\right.$$
$$Sq^2_{C_2}Sq^n_{C_2}=\left\{
\begin{array}{ll}
      Sq^{n+2}_{C_2}+uSq^{n+1}_{C_2}Sq^1_{C_2} &  for \ n\equiv 0 \ mod \ 4\\
      Sq^{n+1}_{C_2}Sq^1_{C_2} &  for \ n\equiv 1 \ mod \ 4\\
      uSq^{n+1}_{C_2}Sq^1_{C_2} &  for \ n\equiv 2 \ mod \ 4\\
      Sq^{n+2}_{C_2}+Sq^{n+1}_{C_2}Sq^1_{C_2} &  for \ n\equiv 3 \ mod \ 4\\
\end{array} 
\right.$$
and
$$Sq^3_{C_2}Sq^n_{C_2}=\left\{
\begin{array}{ll}
      Sq^{n+3}_{C_2}+aSq^{n+1}_{C_2}Sq^1_{C_2} &  for \ n\equiv 0 \ mod \ 4\\
      Sq^{n+2}_{C_2}Sq^1_{C_2} &  for \ n\equiv 1 \ mod \ 4\\
      aSq^{n+1}_{C_2}Sq^1_{C_2} &  for \ n\equiv 2 \ mod \ 4\\
      Sq^{n+2}_{C_2}Sq^1_{C_2} &  for \ n\equiv 3 \ mod \ 4\\
\end{array} 
\right.$$
\end{examples}

Now, let $Sq^I_{C_2}$ denote $Sq^{i_1}_{C_2}Sq^{i_2}_{C_2}\cdots Sq^{i_n}_{C_2}$ for a sequence of integers $I=(i_1, \cdots, i_n)$. The sequence $I$ is said to be \textbf{admissible} if $i_s\geq 2i_{s+1}$ for all $s\geq1$, where $i_{s+1}=0$.

The operations $Sq^I_{C_2}$ with admissible $I$ are called admissible monomials. We also call $Sq^0_{C_2}$ admissible, where $Sq^0_{C_2}=Sq^I_{C_2}$ for empty $I$.
\begin{lemma} The admissible monomials form a basis for the $C_2$-equivariant Steenrod algebra $\mathcal{A}_{C_2}$ as a $H^{\star}(pt)$-module. 
\end{lemma}
\begin{proof}
The proof follows from the Adem relations and the Cartan formula as in the classical case.
\end{proof}

For the graded $\mathcal{A}_{C_2}$-module structure and Hopf algebra structure of equivariant Steenrod algebra, one can look \cite{Nicolas2015}. We will now give unstable module structure of it.

\begin{definition}\label{ex} An $\mathcal{A}_{C_2}$-module is unstable if it satisfies the preceeding instability condition (vii). 
\end{definition}

We define the \textbf{excess} of $Sq^k_{C_2}$ to be the degree of $Sq^{k}_{C_2}$
$$e(Sq^k_{C_2})=\left\{
\begin{array}{ll}
      i\rho &  for \ k=2i\\
      i\rho+1 &  for \ k=2i+1.\\
\end{array} 
\right.$$
So, $e(Sq^k_{C_2})=|Sq^k_{C_2}|$. Then the \textbf{excess} of $Sq^{I}_{C_2}=Sq^{i_1}_{C_2}Sq^{i_2}_{C_2}\cdots Sq^{i_k}_{C_2}$ to be
$$e(Sq^{I}_{C_2})=\sum_{j}e(Sq^{i_j}_{C_2})-\rho e(Sq^{i_{j+1}}_{C_2})$$
where $\rho(r\sigma+s)=(r+s)\rho$.
\begin{examples}\
\begin{itemize}
\item The monomial with $e(Sq^{I}_{C_2})=0$ is $Sq^{0}_{C_2}$.
\item The monomials with $e(Sq^{I}_{C_2})=1$ are $Sq^{1}_{C_2}$, $Sq^2_{C_2}Sq^1_{C_2}$, $Sq^4_{C_2}Sq^2_{C_2}Sq^1_{C_2}$, $\cdots$
\item There is no monomial with $e(Sq^{I}_{C_2})=\sigma$.
\item The monomials with $e(Sq^{I}_{C_2})=2$ are $Sq^3_{C_2}Sq^1_{C_2}$, $Sq^6_{C_2}Sq^3_{C_2}Sq^1_{C_2}$, $Sq^{12}_{C_2}Sq^6_{C_2}Sq^3_{C_2}\-Sq^1_{C_2}$, $\cdots$
\item The monomials with $e(Sq^{I}_{C_2})=\rho$ are $Sq^{2}_{C_2}$, $Sq^4_{C_2}Sq^2_{C_2}$, $Sq^8_{C_2}Sq^4_{C_2}Sq^2_{C_2}$, $\cdots$
\item There is no monomial with $e(Sq^{I}_{C_2})=2\sigma$,
\item The monomials with $e(Sq^{I}_{C_2})=3$ are $Sq^7_{C_2}Sq^3_{C_2}Sq^1_{C_2}$, $Sq^{11}_{C_2}Sq^5_{C_2}Sq^2_{C_2}Sq^1_{C_2}$, $\cdots$
\item The monomials with $e(Sq^{I}_{C_2})=2+\sigma$ are $Sq^3_{C_2}$, $Sq^4_{C_2}Sq^1_{C_2}$, $Sq^5_{C_2}Sq^2_{C_2}$, $Sq^6_{C_2}Sq^3_{C_2}$, $Sq^6_{C_2}Sq^2_{C_2}Sq^1_{C_2}$, $Sq^8_{C_2}Sq^4_{C_2}Sq^1_{C_2}$, $\cdots$
\item There is no monomial with $e(Sq^{I}_{C_2})=1+2\sigma$.
\end{itemize}
\end{examples}

\begin{remark} There is no monomial with $e(Sq^{I}_{C_2})=r\sigma+s$  if $r>s$.
\end{remark}

Let $\mathbf{t}_{j, k}= Sq^{j2^{k-1}}_{C_2}\cdots Sq^j_{C_2}$. Then the set of elements with total excess 1 is $$\left\{ \mathbf{t}_{1, k_1} | k_1 > 0 \right\}.$$\\
The set of elements with total excess 2 is 
$$\left\{ \mathbf{t}_{1+2^{k_1}, k_2+1} \mathbf{t}_{1, k_1} | k_1, k_2 \geq 0 \right\}.$$\\
The set of elements with total excess 3 is 
$$\left\{ \mathbf{t}_{1+2^{k_2}+2^{k_1+k_2}, k_3+1} \mathbf{t}_{1+2^{k_1}, k_2} \mathbf{t}_{1, k_1} | k_1, k_2, k_3 \geq 0 \right\}.$$

The $C_2$-equivariant mod $2$ dual Steenrod algebra (one can check \cite{Nicolas2015}, or \cite{HuKriz1} for details) is
$$\mathcal{A}^{C_2}=M_2^{C_2}[\tau_i, \xi_i]/(\tau_i^2+a\tau_{i+1}\eta_{R}(u)\xi_{i+1})$$
such that
\begin{align*}
& \eta_R(u)= u+a\tau_0\\
& \eta_R(a)= a\\
& |\xi_i|= (2^i-1)\rho\\
& |\tau_i|= 1+|\xi_i|\\
& \Delta(\xi_i)=\sum_{j=0}^i\xi_{i-j}^{2^j}\otimes\xi_j, \ \text{where} \ \xi_0=1\\
& \Delta(\tau_i)=\tau_i\otimes 1+\sum_{j=0}^i\xi_{i-j}^{2^j}\otimes\tau_j.\\ 
\end{align*}

\end{section}

\begin{section}{Equivariant Eilenberg-Mac Lane Spaces}\label{s4}

For each Mackey functor $\underline{M}$, there is an Eilenberg-Mac Lane $G$-spectrum $H\underline{M}$ which has the property as Mackey functors
$$\underline{\pi}_n^G(H\underline{M})=\left\{
\begin{array}{ll}
      \underline{M} &  n=0\\
      0 &  n\in \mathbb{Z}, n\neq0\\
\end{array} 
\right.$$
One can check \cite[Chapter XIII, page 162]{May:Alaska} for the proof of the existence.

Let $\underline{M}$ be a Mackey functor, the $V$th space in the $\Omega$-spectrum for $H\underline{M}$
is called an equivariant Eilenberg-Mac Lane space of type $K(\underline{M},V )$, which is a classifying space for the functor $H^{V}_{G}(-; \underline{M})$. That is, given any real orthogonal representations $V$ , $W$, there is a $G$-homotopy equivalence $K(\underline{M}; V)\simeq \Omega^WK(\underline{M}, V+W)$
satisfying various compatibility properties. Such spaces are constructed in \cite{Lewis1992}, or one can look \cite{DSantos2003}  for a construction with a different method. Here, I will give the definition of them for consistency.
\begin{definition}\cite{Lewis1992}\label{EMS} Let $V$ be a real orthogonal representation with $|V^G|\geq 1$ and $\underline{M}$ be a Mackey functor. An equivariant Eilenberg-Mac Lane space $K(\underline{M}, V)$ is a based, $(|V^*|-1)$-connected $G$-space with the $G$-homotopy type of a $G$-CW complex such
that $\underline{\pi}_V^G(K(\underline{M}, V))=\underline{M}$, and for $\underline{\pi}_{V+k}^G(K(\underline{M}, V))=0$ $k\neq0$.
\end{definition}

\begin{remark} One can ask what $\underline{\pi}_{V+n\sigma}^G(K(\underline{M}, V))$ is for $n>0$. Our main interest is $K(\underline{Z/2}, V)$. Then, $$\underline{\pi}_{V+n\sigma}^{C_2}(K(\underline{Z/2}, V))(C_2/e)= \pi_{V+n\sigma}^{e}(K(\underline{Z/2}, V))=0$$ and
\begin{align*} 
\underline{\pi}_{V+n\sigma}^{C_2}(K(\underline{Z/2}, V))(C_2/C_2) &= \pi_{V+n\sigma}^{C_2}(K(\underline{Z/2}, V))  \\ 
 &\cong \tilde{H}_{V+n\sigma}^{C_2}(S^V; \underline{Z/2})\\
 &\cong \tilde{H}_{n\sigma}^{C_2}(S^{0, 0}; \underline{Z/2})\\
 &\cong H_{n\sigma}^{C_2}(\ast; \underline{Z/2})
\end{align*}
So, as a Mackey functor, the homotopy $\underline{\pi}_{V+n\sigma}^G(K(\underline{M}, V))$ is one of the
\begin{displaymath}
\xymatrix
@R=7mm
@C=10mm{
Z/2 \ar@/_/[d]_{}\\
0 \ar@(dr,dl)[]^{Id} \ar@/_/[u]_{}\\
}
\
\
\
\text{or}
\
\
\
\xymatrix
@R=7mm
@C=10mm{
0 \ar@/_/[d]_{}\\
0 \ar@(dr,dl)[]^{Id} \ar@/_/[u]_{}\\
}
\end{displaymath}
depending on the dimension of the representation $V$ and $n$.
\end{remark}

As mentioned before, one can check \cite{Lewis1992} for existence and some properties of these spaces.

Another approach to construct equivariant Eilenberg-Mac Lane spaces is Dos Santos \cite{DSantos2003} approach. As we know in the classical case, the free abelian group on the $n$-sphere is a model for the Eilenberg-Mac Lane space $K(\mathbb{Z},n)$, and the free $\mathbb{F}_2$-vector space on the $n$-sphere is a model for the Eilenberg-Mac Lane space $K(\mathbb{F}_2,n)$. Dos Santos constructed a topological abelian group $M\otimes X$ in \cite[Definition 2.1.]{DSantos2003}, which is the equivariant generalization of previous sentence for a Mackey functor $M$, and proved an $RO(G)$-graded version of equivariant Dold-Thom theorem proved by Lima-Filho for $\mathbb{Z}$-graded case in \cite{Lima-Filho1997}.

Let $M$ be a $\mathbb{Z}[G]$-module, $\underline{M}$ be the Mackey functor associated to $M$: the value of $\underline{M}$ on $G/H$ is $M^H$ and the value on the
projection $G/K\longrightarrow G/H$, for $K < H < G$, is the inclusion of $M^H\hookrightarrow M^K$. We define $M\otimes X$ as the $\mathbb{Z}[G]$-module with a topology as follows(\cite[Definition 2.1.]{DSantos2003}):
Let $(X,\ast)$ be a based $G$-set, $M\otimes X$ denote the $\mathbb{Z}[G]$-module $\bigoplus_{x\in X-\{\ast\}}M$. The action  of $g\in G$ is given by $(g.m)_x=g.m_{g^{-1}.x}$, where $m_x$ denotes the $x$th coordinate of $m\in \bigoplus_{x\in X-\{\ast\}}M$. Given $(X,\ast)$ a based $G$-space, $M\otimes X$ can be equivalently defined as the quotient 
$$M\otimes X= \amalg_{n\geq0} M^n\times X^n / \backsim ,$$
where $\backsim$ is the equivalence relation generated by:
\begin{enumerate}
\item[(i)] $(r,\phi^{\ast}x)\backsim (\phi_{\ast}r,x)$, for each based map $\phi: \{0, \cdots, n\}\longrightarrow \{0, \cdots, m\}$, $n, m\in \mathbb{N}$, where $\phi^{\ast}x=x\circ \phi$, and $(\phi_{\ast}r)_i=\sum_{k\in \phi^{-1}(i)}r_k$.
\item[(ii)] $((r,r^{\prime}), (x, \ast))\backsim (r, x)$, for each $r\in M^n$, $r^{\prime}\in M$, $x\in X$.
\end{enumerate}
We give the discrete topology to $M$ and endow $M\otimes X$ with the quotient
topology corresponding to the relation $\backsim$.

We can define Eilenberg-Mac Lane spaces as $K_V= M\otimes S^V$. In our case,
$$K_{m+n\sigma}= \mathbb{Z}/2\otimes S^{m+n\sigma}.$$

\begin{theorem}\cite{DSantos2003} Let $X$ be a based $G$-CW-complex and let $V$ be a finite dimensional
$G$-representation, then $M\otimes X$ is an equivariant infinite loop space and there is a
natural equivalence $$\pi_V^G(M\otimes X)\cong \tilde{H}_V^G(X; \underline{M})$$
\end{theorem}
As a corollary to this theorem we have that $M\otimes S^V$ is a $K(\underline{M}, V)$ space (as Definition \ref{EMS}). Thus
we have a simple model for the equivariant Eilenberg-Mac Lane spectrum $H\underline{M}$.

\begin{examples}\
\begin{itemize}
\item[(i)] $K(\underline{Z/2}, 1)$ is $RP^{\infty}$, with trivial action.
\item[(ii)] Recall that $RP^{\infty}_{tw}=\mathbb{P}(\mathcal{U})$ is the space of lines in the
complete universe (Definition \ref{Uni})
$$\mathcal{U}=(\mathbb{R}^{\rho})^{\infty}$$
\cite{May:Alaska}. The cohomology of $RP^{\infty}_{tw}$ is calculated by Kronholm in \cite{WKronholm2010}. The space $RP^{\infty}_{tw}$ is equivalent to $K(\underline{Z/2}, \sigma)$, since it is  equivalent to $\mathbb{Z}/2\otimes S^{\sigma}$.   
\end{itemize}
\end{examples}
\begin{theorem}\cite{W2010} $H^{\star}(RP^{\infty}_{tw})\cong H^{\star}(pt)[c, d]/(c^2=ac+ud)$, where deg($c$)=$\sigma$, and deg($d$)=$\rho$.
\end{theorem}

Now, we will give a structure of fixed points of equivariant Eilenberg-Mac Lane spaces, which is useful to calculate the cohomology of them. 
\begin{theorem}\label{1}\cite[Corollary 10]{JLCaruso1999}
\begin{itemize}
\item[(i)]$(K(\underline{Z/2}, r\sigma+s))^e\simeq K(Z/2, r+s)$.
\item[(ii)]$(K(\underline{Z/2}, r\sigma+s))^{C_2}\simeq K(Z/2, s)\times\cdots \times K(Z/2, r+s)$.
\end{itemize}
\end{theorem}

\end{section}

\begin{section}{Cohomology of Eilenberg-Mac Lane Spaces}\label{s5}

In classical case the cohomology of Eilenberg- Mac Lane spaces $K_n$ with $\mathbb{Z}/2$-coefficients, which is given by Serre in \cite{Serre} is a polynomial ring
$$H^{\ast}(K_n ; \mathbb{Z}/2)= P(Sq^{I}(\iota_n)| e(I)<n)$$
where $I$ are admissible sequences, $\iota_n$ is the fundamental class, and $e(Sq^{I})=\sum_{j}(i_j-2i_{j+1})$. We thought that we can give similar description for $RO(C_2)-$graded $C_2$-equivariant cohomology of $C_2$-equivariant Eilenberg-Mac Lane spaces, but these are more complicated than we expect.

Let $\mathbf{s}_{V, l}$ is the operation that sends $x$ to $x^{2^l}$ for $x \in H^V$. It is possible to express $\mathbf{s}_{V, l}$ as a linear combination of Steenrod operations.
$$\mathbf{s}_{V, 0}=1$$
If $x\in H^{a+b\sigma}$, and $b= r_1+ \lfloor \frac{a+b}{2} \rfloor$, then $(u^{-r_1}x)^2=Sq_{C_2}^{a+b}(u^{-r_1}x)$, so
$$x^2=u^{2r_1}Sq_{C_2}^{a+b}(u^{-r_1}x)$$
By using $C_2$-equivariant Cartan formula and the formula \ref{act2} 
$$
Sq^{2m+\delta}_{C_2}(u^{2l+\epsilon})= \binom{2l+\epsilon}{2m+\delta} u^{2l+\epsilon-m-\delta}a^{2m+\delta}
$$
one has general formula for $Sq_{C_2}^{a+b}(u^{-r_1}x)$.
By iterating this method one can find a formula for every $x^{2^l}$, so $\mathbf{s}_{V,l}$ exist. For example, if $x\in H^{3+\sigma}$, then
\begin{align*}
(ux)^2=Sq^4_{C_2}(ux)&=\sum_{r=0}^{2}Sq^{2r}_{C_2}(u)Sq^{4-2r}_{C_2}(x)+\sum_{s=0}^{1}Sq^{2s+1}_{C_2}(u)Sq^{3-2s}_{C_2}(x)\\
&=uSq^4_{C_2}(x)+uaSq^3_{C_2}(x)
\end{align*}

Thus
$$x^2=u^{-1}Sq^4_{C_2}(x)+u^{-1}aSq^{3}_{C_2}(x).$$

The set of elements $x_i$ whose finite distinct products form a basis for a graded ring $A$ is called a \textbf{simple
system of generators}. For example, a polynomial algebra $k[x]$ has a simple system of generators $\left\{ x^{2^i} | \ i\geq 0 \right\}$.
\begin{theorem}\label{Bor}(Borel) Let $F\rightarrow E\rightarrow B$ be a $C_2$-fibration with $E$ contractible. Suppose that $H^{\star}(F)$ has a simple system $\{x_i\}$ of transgressive generators. Then $H^{\star}(B)$ is a polynomial ring in the $\left\{\Sigma(x_i) \right\}$.
\end{theorem}
$E_2$-page of $RO(G)$-graded Serre spectral sequence of Kronholm \cite{WKronholm2010-2} depends only on the total degree of representations, not the dimension of twisted part. The proof of the theorem is completely same as the classical case. See, for example, \cite[Page 88, Theorem 1]{MR0226634}.

A simple system of generators for $H^{\star}(K_{\sigma})\cong H^{\star}(pt)[c, d]/(c^2=ac+ud)$ is $$\left\{ c, d^{2^l} | l\geq 0 \right\}=\left\{c, \mathbf{s}_{1+\sigma,l}(d)| l\geq 0 \right\}$$

By applying the Borel Theorem to the path space fibration
$$K_{\sigma}\rightarrow P(K_{\rho})\rightarrow K_{\rho}$$
we have
$$ H^{\star}(K_{\rho})=P(x_{\rho}, \mathbf{s}_{\rho,l}(x_{1+\rho}) | l\geq 0).$$

A simple system of generator for $H^{\star}(K_{\rho})$ is
$$\left\{ x_{\rho}^{2^j}, (\mathbf{s}_{\rho,l}(x_{1+\rho}))^{2^j} | j,l \geq 0 \right\}=\left\{ \mathbf{s}_{\rho, j}(x_{\rho}), \mathbf{s}_{2^l\rho+1,j}\mathbf{s}_{\rho,l}(x_{1+\rho}) | j, l\geq 0 \right\}$$
where $|\mathbf{s}_{\rho,l}(x_{1+\rho})|= 2^l\rho+1$. Then,
$$ H^{\star}(K_{1+\rho})=P(\mathbf{s}_{\rho,j}(x_{1+\rho}),  \mathbf{s}_{2^l\rho+1, j} \mathbf{s}_{\rho, l}(x_{2+\rho}) | j, l\geq 0).$$

A simple system of generators for $H^{\star}(K_{1+\rho})$ is
\begin{align*}
& \left\{ (\mathbf{s}_{\rho, j} x_{1+\rho})^{2^k}, (\mathbf{s}_{2^j\rho+1, j} \mathbf{s}_{\rho,l} x_{2+\rho})^{2^k} | j,l, k \geq 0 \right\}\\
& =\left\{ \mathbf{s}_{2^j\rho+1, k} \mathbf{s}_{\rho, j} x_{2+\sigma}, \mathbf{s}_{2^j(2^l\rho+1)+1, k} \mathbf{s}_{2^l\rho+1,j} \mathbf{s}_{\rho,l} x_{2+\rho} | j, l, k \geq 0 \right\}
\end{align*}
where $|\mathbf{s}_{2^l\rho+1, j} \mathbf{s}_{\rho,l} x_{2+\rho}|= 2^j(2^l\rho+1)+1$. Then,
$$ H^{\star}(K_{2+\rho})=P( \mathbf{s}_{2^j\rho+1, k} \mathbf{s}_{\rho, j} x_{2+\rho}, \mathbf{s}_{2^j(2^l\rho+1)+1, k} \mathbf{s}_{2^l\rho+1,j}\mathbf{s}_{\rho,l} x_{3+\rho} | j, l, k \geq 0).$$
Thus, by iterating this process, one can find the $RO(C_2)$-graded cohomology of $C_2$-equivariant Eilenberg-Mac Lane spaces $K_{n+\sigma}$
$$H^{\star}(K_{n+\sigma})$$
for $n\geq 0$.
\begin{conjecture}[]{}
We know that
\begin{align*}
&H^{\ast}(K_1)=P(x_1)\\
&H^{\ast}(K_{\sigma})=P(x_{\sigma}, x_{1+\sigma})/(x_{\sigma}^{2}+ax_{\sigma}+ux_{1+\sigma}).\\
\end{align*}
If we knew $H^{\star}(K_{1+k(\sigma-1)})$ for all $k\geq 0$, then we could use the Borel theorem to find $H^{\star}(K_{1+m+k(\sigma-1)})$ for $m\geq 0$.

We conjectured that $H^{\star}(K_{1+k(\sigma-1)})$ is a polynomial algebra on $k+1$ generators with $k$ relations, saying that the square of each of the first k generators is a linear combination of the other generators. The dimensions of the first $k$ generators of the $H^{\star}(K_{1+k(\sigma-1)})$ are obtained by adding $\sigma-1$ to those of the generators of the $H^{\star}(K_{1+(k-1)(\sigma-1)})$, and the dimension of the last generator is $k\sigma+2^k-k$.
\end{conjecture}

\begin{example} Let $k=3$. Then
$$H^{\ast}(K_{3\sigma-2})=P(x_{3\sigma-2}, x_{3\sigma-1}, x_{3\sigma+1}, x_{3\sigma+5})/(x_{3\sigma-2}^2+\cdots,x_{3\sigma-1}^2+\cdots,x_{3\sigma+1}^2+\cdots)$$
where the other terms in the relations are linear. The resulting simple system of generators is
$$\{x_{3\sigma-2}, x_{3\sigma-1}, x_{3\sigma+1}\}\cup \{x_{3\sigma+5}^{2^{i}}| i\geq 0\}.$$
\end{example}

Now, we give another useful lemma for computations, which is the cohomology analogous of the Lemma 2.7. in \cite{Wil-Beh}.

There is a forgetful map $$\Phi^{e}: H^{V}_{C_2}(X; \underline{Z/2})\longrightarrow H^{|V|}(X^e; Z/2)$$
from the equivariant cohomology to the non-equivariant cohomology with $Z/2$-coefficients. And also, we have a fixed point map 
$$\Phi^{C_2}: H^{V}_{C_2}(X; \underline{Z/2})\longrightarrow H^{|V^{C_2}|}(X^{\Phi C_2}; Z/2)$$
where $X^{\Phi C_2}$ is a geometric fixed point of a $G$-space $X$. Now, I will state the lemma, whose proof is the analog of Lemma 2.7. in \cite{Wil-Beh}.

\begin{lemma}\label{2} Let $X$ be a genuine $C_2$-spectrum, and suppose that $\{b_i\}$ is a set of elements of $H^{\star}(X)$ such that
\begin{itemize}
\item[(i)]$\{\Phi^e(b_i)\}$ is a basis of $H^{\ast}(X^e)$, and
\item[(ii)]$\{\Phi^{C_2}(b_i)\}$ is a basis of $H^{\ast}(X^{\Phi C_2})$
\end{itemize}
Then $H^{\star}(X)$ is free over $H^{\star}(pt)$ with the basis $\{b_i\}$.
\end{lemma}

One project is to finish calculations of the $RO(C_2)-$graded $C_2$-equivariant cohomology of $C_2$-equivariant Eilenberg-Mac Lane spaces by using Caruso theorem \ref{1} and lemma \ref{2}, and then Eilenberg-Moore spectral sequences of Michael A. Hill \cite[Chapter 5]{Hillfree}.

\begin{conjecture} $H^{\star}(K_{r\sigma+s})$ is a polynomial algebra on certain $C_2$-equivariant Steenrod operations $Sq_{C_2}^{I}(\iota_{r\sigma+s})$ divided by certain powers of $u$, where $e(I)<r\sigma+s$, and $\iota_{r\sigma+s}$ is the fundamental class, and $V<V^{'}$ if and only if $V^{'}=V+W$ for some actual representations $W$ with positive degree.
\end{conjecture}

\begin{example}$H^{\star}(K(\underline{Z/2}, 1+\sigma))$ is the polynomial algebra generated by elements $Sq^I(\iota_{1+\sigma)}$, where $I$ is admissible and $e(I)<1+\sigma$. So, it is a polynomial algebra
$$P(Sq^0(\iota_{1+\sigma}), Sq^2Sq^1(\iota_{1+\sigma}), Sq^4Sq^2Sq^1(\iota_{1+\sigma}),\cdots).$$
Then, it is shortly
$$P(x_{\rho}, x_{1+2\rho}, x_{1+4\rho}, x_{1+8\rho}, \cdots).$$
\end{example}

By now, we have calculated the cohomology of $K(\underline{Z/2}, n+\sigma)$ for $n\geq 0$. To calculate other cases, if knew $H^{\ast}(K_{n\sigma})$ for $n\geq 2$, we could use the Eilenberg-Moore spectral sequence \cite[Chapter 5]{Hillfree} anf the Borel theorem for the path-space fibration
$$\Omega K(\underline{Z/2}, V)\longrightarrow P(K(\underline{Z/2}, V))\longrightarrow K(\underline{Z/2}, V).$$

For example, for the path-space fibration
$$K(\underline{Z/2}, \sigma)\longrightarrow P(K(\underline{Z/2}, 1+\sigma))\longrightarrow K(\underline{Z/2}, 1+\sigma).$$
\end{section}
$E_{\infty}$-term of the Eilenberg-Moore spectral sequence is $$E_{\infty}=E(x_{\sigma}, x_{\rho}, x_{2\rho}, x_{4\rho}, \cdots)$$
with the relations
$$x_{\sigma}^2=ax_{\sigma}+ux_{1+\sigma}$$
$$x_{1+\sigma}^2=x_{2+2\sigma}$$
$$x_{2+2\sigma}^2=x_{4+4\sigma}$$
$$\vdots$$
$$x^2_{2^i\rho}=x_{2^{i+1}\rho}$$
$$\vdots$$
for $i\geq 0$. As a result, $H^{\star}(K(\underline{Z/2}, \sigma), \underline{Z/2})$ is a quadratic extension of a polynomial algebra, as it has already known.

\bibliographystyle{plain}
\bibliography{math}

\begin{thebibliography}{10}

\bibitem{Wil-Beh}
Mark Behrens and Dylan Wilson.
\newblock A $c_2$-equivariant analog of mahowald's thom spectrum theorem.
\newblock {\em Proceedings of the American Mathematical Society}, 146, 07 2017.

\bibitem{JLCaruso1999}
Jeffrey~L. Caruso.
\newblock Operations in equivariant {${\bf Z}/p$}-cohomology.
\newblock {\em Math. Proc. Cambridge Philos. Soc.}, 126(3):521--541, 1999.

\bibitem{DSantos2003}
Pedro~F. dos Santos.
\newblock A note on the equivariant {D}old-{T}hom theorem.
\newblock {\em J. Pure Appl. Algebra}, 183(1-3):299--312, 2003.

\bibitem{Hillfree}
Michael~A. Hill.
\newblock Freeness and equivariant stable homotopy.
\newblock {\em Journal of Topology}, 15(2):359--397.

\bibitem{HuKriz1}
Po~Hu and Igor Kriz.
\newblock Real-oriented homotopy theory and an analogue of the
  {A}dams-{N}ovikov spectral sequence.
\newblock {\em Topology}, 40(2):317--399, 2001.

\bibitem{W2010}
William~C. Kronholm.
\newblock A freeness theorem for {$RO(\Bbb Z/2)$}-graded cohomology.
\newblock {\em Topology Appl.}, 157(5):902--915, 2010.

\bibitem{WKronholm2010-2}
William~C. Kronholm.
\newblock The {${\rm RO}(G)$}-graded {S}erre spectral sequence.
\newblock {\em Homology Homotopy Appl.}, 12(1):75--92, 2010.

\bibitem{LMM1981}
G.~Lewis, J.~P. May, and J.~McClure.
\newblock Ordinary {$RO(G)$}-graded cohomology.
\newblock {\em Bull. Amer. Math. Soc. (N.S.)}, 4(2):208--212, 1981.

\bibitem{Lewis1992}
L.~Gaunce Lewis, Jr.
\newblock Equivariant {E}ilenberg-{M}ac {L}ane spaces and the equivariant
  {S}eifert-van {K}ampen and suspension theorems.
\newblock {\em Topology Appl.}, 48(1):25--61, 1992.

\bibitem{Lewis92}
L.~Gaunce Lewis, Jr.
\newblock The equivariant {H}urewicz map.
\newblock {\em Trans. Amer. Math. Soc.}, 329(2):433--472, 1992.

\bibitem{Lima-Filho1997}
P.~Lima-Filho.
\newblock On the equivariant homotopy of free abelian groups on {$G$}-spaces
  and {$G$}-spectra.
\newblock {\em Math. Z.}, 224(4):567--601, 1997.

\bibitem{Mah:ImJ}
Mark Mahowald.
\newblock The image of {$J$} in the {$EHP$} sequence.
\newblock {\em Ann. of Math. (2)}, 116(1):65--112, 1982.

\bibitem{May:Alaska}
J.~P. May.
\newblock {\em Equivariant homotopy and cohomology theory}, volume~91 of {\em
  CBMS Regional Conference Series in Mathematics}.
\newblock Published for the Conference Board of the Mathematical Sciences,
  Washington, DC, 1996.
\newblock With contributions by M. Cole, G. Comeza{\~n}a, S. Costenoble, A. D.
  Elmendorf, J. P. C. Greenlees, L. G. Lewis, Jr., R. J. Piacenza, G.
  Triantafillou, and S. Waner.

\bibitem{MilA}
John Milnor.
\newblock The {S}teenrod algebra and its dual.
\newblock {\em Ann. of Math. (2)}, 67:150--171, 1958.

\bibitem{MR0226634}
Robert~E. Mosher and Martin~C. Tangora.
\newblock {\em Cohomology operations and applications in homotopy theory}.
\newblock Harper \& Row, Publishers, New York-London, 1968.

\bibitem{Nicolas2015}
Nicolas Ricka.
\newblock Subalgebras of the {$\Bbb Z/2$}-equivariant {S}teenrod algebra.
\newblock {\em Homology Homotopy Appl.}, 17(1):281--305, 2015.

\bibitem{1207.3121}
Jo\"el Riou.
\newblock Op\'erations de {S}teenrod motiviques.
\newblock https://arxiv.org/abs/1207.3121, 2012.

\bibitem{Serre}
Jean-Pierre Serre.
\newblock {\em Repr\'esentations lin\'eaires des groupes finis}.
\newblock Hermann, Paris, 1967.

\bibitem{Voe:Red}
Vladimir Voevodsky.
\newblock Reduced power operations in motivic cohomology.
\newblock {\em Publ. Math. Inst. Hautes \'Etudes Sci.}, (98):1--57, 2003.

\end{thebibliography}
\end{document}